\documentclass[12pt, reqno]{amsart}
\usepackage{amsmath, amsthm, amscd, amsfonts, amssymb, graphicx}
\usepackage[bookmarksnumbered,plainpages]{hyperref}

\newtheorem{theorem}{Theorem}[section]
\newtheorem{lemma}[theorem]{Lemma}

\newtheorem{corollary}[theorem]{Corollary}
\theoremstyle{definition}

\theoremstyle{remark}
\newtheorem{remark}[theorem]{Remark}

\numberwithin{equation}{section}

\begin{document}
\title[An operator equality and its norm inequalities]
{An operator equality involving a continuous field of operators and
its norm inequalities}

\author[M.S. Moslehian and F. Zhang]{Mohammad Sal Moslehian and Fuzhen Zhang}
\address{Department of Mathematics, Ferdowsi University of Mashhad,
P.O. Box 1159, Mashhad 91775, Iran;
\newline Centre of Excellence in Analysis on
Algebraic Structures (CEAAS), Ferdowsi University of Mashhad, Iran.}
\email{moslehian@ferdowsi.um.ac.ir and moslehian@ams.org}
\address{Farquhar College of Arts and Sciences, Nova Southeastern
University, 3301 College Ave, Fort Lauderdale, FL 33314,
USA;\newline College of Mathematics and Systems Science, Shenyang
Normal University, Shenyang, Liaoning Province 110034,  China.}
\email{zhang@nova.edu}



\begin{abstract} Let ${\mathfrak A}$ be a $C^*$-algebra, $T$ be a locally compact
Hausdorff space equipped with a probability measure $P$ and let
$(A_t)_{t\in T}$ be a continuous field of operators in ${\mathfrak
A}$ such that the function $t \mapsto A_t$ is norm continuous on $T$
and the function $t \mapsto \|A_t\|$ is integrable. Then the
following equality including Bouchner integrals holds
\begin{eqnarray}\label{oi}
\int_T\left|A_t - \int_TA_s{\rm d}P\right|^2 {\rm
d}P=\int_T|A_t|^2{\rm d}P - \left|\int_TA_t{\rm d}P\right|^2\,.
\end{eqnarray}
This equality is related both to the notion of variance in
statistics and to a characterization of inner product spaces. With
this operator equality, we present some uniform norm and Schatten
$p$-norm inequalities.
\end{abstract}
\maketitle

%--------------------------------------------------------------------------------------------------%

\section{Introduction and preliminaries}

Many interesting characterizations of inner product spaces have been
introduced (see, e.g., \cite{AMI}). It is shown by Th. M. Rassias
\cite{RAS} that a normed space $X$ (with norm $\|\cdot \|$) is an
inner product space if and only if for any finite set of vectors
$x_1, \cdots, x_n\in X$,
\begin{eqnarray}\label{0}
\sum_{i=1}^n\Big\|x_i-\frac{1}{n}\sum_{j=1}^n x_j\Big\|^2=
\sum_{i=1}^n\|x_i\|^2 - n\Big\|\frac{1}{n}\sum_{j=1}^nx_j\Big\|^2\,.
\end{eqnarray}

This equality is of fundamental importance in the study of normed
spaces and inner product spaces since it reveals a basic relation
between the two sorts of spaces. From statistical point of view, let
$(X, \mu)$ be a probability measure space and $f$ be a random
variable, i.e., $f$ is an element of $L^2(X,\mu)$. With the variance
of $f$ defined by
$$\textrm{Var}(f)=E(|f-E(f)|^2),$$
where $E(f)=\int_X f{\rm d}\mu$ denotes the expectation of $f$,
\eqref{0} resembles (or vice versa) the well-known equality
$$\textrm{Var}(f)=E(|f|^2)-|E(f)|^2.$$

Let ${\mathbb B}({\mathcal H})$ be the algebra of all bounded linear
operators on a separable complex Hilbert space ${\mathcal H}$
endowed with inner product $\langle \cdot , \cdot \rangle $ and the
operator norm $\|\cdot\|_\infty$. We denote the absolute value of $A
\in {\mathbb B}({\mathcal H})$ by $|A|=(A^*A)^{1/2}$. For $x, y \in
{\mathcal H}$, the rank one operator $x \otimes y$ is defined on
${\mathcal H}$ by $(x\otimes y)(z)=\langle z,y\rangle x$.

Let ${\mathfrak A}$ be a $C^*$-algebra and let $T$ be a locally
compact Hausdorff space. A field $(A_t)_{t\in T}$ of operators in
${\mathfrak A}$ is called a continuous field of operators if the
function $t \mapsto A_t$ is norm continuous on $T$. If $\mu(t)$ is a
Radon measure on $T$ and the function $t \mapsto \|A_t\|$ is
integrable, one can form the Bochner integral $\int_{T}A_t{\rm
d}\mu(t)$, which is the unique element in ${\mathfrak A}$ such that
$$\varphi\left(\int_TA_t{\rm d}\mu(t)\right)=\int_T\varphi(A_t){\rm d}\mu(t)$$
for every linear functional $\varphi$ in the norm dual ${\mathfrak
A}^*$ of ${\mathfrak A}$; cf. \cite[Section 4.1]{H-P}.

Let $A\in {\mathbb B}({\mathcal H})$ be a compact operator and let
$0 < p < \infty$. The Schatten $p$-norm ($p$-quasi-norm) for $1 \leq
p < \infty$ $(0< p < 1)$ is defined by $\|A\|_p=(\textrm{tr}
|A|^p)^{1/p}$, where $\textrm{tr}$ is the usual trace functional.
Clearly, for $p, \, q>0$,
$$\|A\|_p^q= \left\|\,|A|^2\,\right\|_{p/2}^{q/2}\,.$$
\noindent For $p>0$, the Schatten $p$-class, denoted by $C_p$, is
defined to be the two-sided ideal in ${\mathbb B}({\mathcal H})$ of
those compact operators $A$ for which $\|A\|_p$ is finite. In
particular, $C_1$ and $C_2$ are the trace class and the
Hilbert-Schmidt class, respectively. For more information on the
theory of the Schatten $p$-classes, the reader is referred to
\cite{BHA, SIM}.

Since $C_2$ is a Hilbert space under the inner product $\langle A,
B\rangle=\textrm{tr} (B^*A)$, it follows from \eqref{0} that if
$A_1, \cdots, A_n \in C_2$, then
\begin{eqnarray}\label{4}
\sum_{i=1}^n\Big\|A_i-\frac{1}{n}\sum_{j=1}^n A_j\Big\|_2^2=
\sum_{i=1}^n\|A_i\|_2^2 -
n\Big\|\frac{1}{n}\sum_{j=1}^nA_j\Big\|_2^2\,.
\end{eqnarray}

In this paper, we establish a general operator version of equality
\eqref{0}, from which we deduce an extension of \eqref{0}. We
present some inequalities concerning various norms such as Schatten
$p$-norms that form natural generalizations (in inequalities)  of
the identity \eqref{4} (see, e.g., \cite{H-K-M}). It seems that the
inequalities related to Schatten $p$-norms are useful in operator
theory and mathematical physics and are interesting in their own
right.

%--------------------------------------------------------------------------------------------------%
\section{Main results}

We begin by establishing an operator version of equality \eqref{0}
involving continuous fields of operators and integral means of
operators.

\begin{theorem}\label{th1}
Let ${\mathfrak A}$ be a $C^*$-algebra, $T$ be a locally compact
Hausdorff space equipped with a probability measure $P$ and let
$(A_t)_{t\in T}$ be a continuous field of operators in ${\mathfrak
A}$ such that the function $t \mapsto A_t$ is norm continuous on $T$
and the function $t \mapsto \|A_t\|$ is integrable. Then
\begin{eqnarray}\label{int}
\int_T\left|A_t - \int_TA_s{\rm d}P(s)\right|^2 {\rm
d}P(t)=\int_T|A_t|^2{\rm d}P(t) - \left|\int_TA_t{\rm
d}P(t)\right|^2
\end{eqnarray}
\end{theorem}
\begin{proof}
\begin{eqnarray*}
\int_T\Big|A_t&-&\int_TA_s{\rm d}P(s)\Big|^2 {\rm d}P(t)\\
&=& \int_T\left(A_t - \int_TA_s{\rm d}P(s)\right)^*\left(A_t -
\int_TA_r{\rm d}P(r)\right) {\rm d}P(t)\\
&=&\int_T|A_t|^2{\rm d}P(t)-\int_T\left(A_t^* \int_T A_r{\rm d}P(r)\right){\rm d}P(t)\\
&&- \int_T\left(\Big(\int_T A_s{\rm d}P(s)\Big)^*A_t\right){\rm d}P(t)\\
&&+ \left(\int_T A_s{\rm d}P(s)\right)^*\int_T A_r{\rm d}P(r)\int_T{\rm d}P(t)\\
&=&\int_T|A_t|^2{\rm d}P(t)- \left(\int_T A_s{\rm d}P(s)\right)^*\int_T A_r{\rm d}P(r)\\
&& \qquad \Big({\rm by~} \int_T{\rm d}P(t)=1 {\rm \,\,~and~} \int_T A_t^*{\rm d}P(t)=\big(\int_T A_t{\rm d}P(t)\big)^*\Big)\\
&=&\int_T|A_t|^2{\rm d}P(t)- \left|\int_T A_t{\rm d}P(t)\right|^2\,.
\end{eqnarray*}
\end{proof}

%--------------------------------------------------------------------------------------------------%

\begin{corollary}\label{cormain}
Let $A_1, \cdots, A_n \in {\mathbb B}({\mathcal H})$. Then

{\rm (i)} for any set of nonnegative numbers $t_1, \cdots, t_n$ with
$\sum_{i=1}^nt_i=1$,
\begin{eqnarray}\label{oit}
\sum_{i=1}^nt_i\left|A_i-\sum_{j=1}^n t_jA_j\right|^2=
\sum_{i=1}^nt_i|A_i|^2 -\left|\sum_{j=1}^nt_jA_j\right|^2\,;
\end{eqnarray}
{\rm (ii)}
\begin{eqnarray}\label{oin}
\sum_{i=1}^n\left|A_i-\frac{1}{n}\sum_{j=1}^n A_j\right|^2=
\sum_{i=1}^n|A_i|^2 - n\left|\frac{1}{n}\sum_{j=1}^nA_j\right|^2\,.
\end{eqnarray}
\end{corollary}
\begin{proof}
(i) Take $T=\{1, \cdots, n\}$ and $P(\{i\})=t_i$ in Theorem
\ref{th1}.

\noindent (ii) Put $t_i=1/n$ in (i).
\end{proof}
As an immediate consequence of \eqref{oit}, we get the following
known AM-QM operator inequality; cf. \cite[Theorem 7]{ZHA}.

\begin{corollary}
Let $A_1, \cdots, A_n \in {\mathbb B}({\mathcal H})$. Then for any
set of nonnegative numbers $t_1, \cdots, t_n$ with
$\sum_{i=1}^nt_i=1$,
\begin{eqnarray*}
\left|\sum_{i=1}^nt_iA_i\right|^2 \leq \sum_{i=1}^nt_i|A_i|^2 \,.
\end{eqnarray*}
\end{corollary}

The following result compares the mean of the squares of the
operators to the square of the mean, and  gives some bounds of their
difference.

\begin{corollary}
Let $A_1, \cdots, A_n\in {\mathbb B}({\mathcal H}) $ be positive
operators such that $0 \leq m_i I \leq A_i \leq M_i I$ for some
nonnegative scalars $m_i$ and $M_i$ and all $1 \leq i \leq n$, and
let $t_1, \cdots, t_n$ be nonnegative numbers such that
$\sum_{i=1}^nt_i=1$. Then
\begin{eqnarray*}
\sum_{i=1}^nt_i\beta_i^2I \leq \sum_{i=1}^nt_iA_i^2 -
\Big(\sum_{i=1}^nt_iA_i\Big )^2 \leq \sum_{i=1}^nt_i\alpha_i^2I\,,
\end{eqnarray*}
where $$\alpha_i=\max\big \{\left|M_i-\sum_{i=1}^nt_im_i\right|,\;
\left|m_i-\sum_{i=1}^nt_iM_i\right|\big \}$$ and $$\beta_i=\min\big
\{\left|M_i-\sum_{i=1}^nt_im_i\right|,\;
\left|m_i-\sum_{i=1}^nt_iM_i\right|\big \}.$$
\end{corollary}
\begin{proof}
It is sufficient to notice that for each $A_i$, by the functional
calculus,
\begin{align*}
\min\Big \{\Big|M_i-\sum_{i=1}^nt_im_i\Big|^2&,\;
\Big|m_i-\sum_{i=1}^nt_iM_i\Big|^2\Big \} \leq
\Big|A_i-\sum_{j=1}^nt_jA_j\Big|^2\\&\leq \max\Big
\{\Big|M_i-\sum_{i=1}^nt_im_i\Big|^2,\;
\Big|m_i-\sum_{i=1}^nt_iM_i\Big|^2\Big \}
\end{align*}
and use \eqref{oit}.
\end{proof}

%--------------------------------------------------------------------------------------------------%

By a $Q$-norm on a subspace ${\mathcal D}$ of ${\mathbb B}({\mathcal
H})$ we mean a norm $\|\cdot\|_Q$ for which there exists a norm
$\|\cdot\|_{\widehat{Q}}$ defined on a subspace $\widehat{{\mathcal
D}}$ such that $\|A\|_Q^2=\|A^*A\|_{\widehat{Q}}\,\,(A\in {\mathcal
D})$ (we implicitly assume that $A^*A\in \widehat{{\mathcal D}}$ if
and only if $A \in {\mathcal D}$). The operator norm
$\|\cdot\|_\infty$ and $\|\cdot\|_p$ (for $2 \leq p < \infty$) are
examples of $Q$-norms on ${\mathbb B}({\mathcal H})$; see, e.g.,
\cite[p.~89]{BHA} or \cite{B-K2}.

\begin{corollary}
Let $\|\cdot\|_Q$ be a $Q$-norm on a subspace ${\mathcal D}$ of
${\mathbb B}({\mathcal H})$, and let $t_1, \cdots, t_n$ be
nonnegative numbers such that $\sum_{i=1}^nt_i=1$. Then for any
$A_1, \cdots, A_n \in {\mathcal D}$ with $A_i^*A_j=0$ for $1 \leq
i\neq j\leq n$,
\begin{eqnarray*}
\Big \|\sum_{i=1}^n\sqrt{t_i}A_i\Big\|_Q^2 \leq \sum_{i=1}^n
t_i\Big\|A_i-\sum_{j=1}^n t_jA_j\Big\|_Q^2
+\Big\|\sum_{j=1}^nt_jA_j\Big\|_Q^2\,.
\end{eqnarray*}
\end{corollary}
\begin{proof}
\begin{eqnarray*}
\Big\|\sum_{i=1}^n \sqrt{t_i} A_i\Big\|_Q^2&=&\Big\|\,\Big |
\sum_{i=1}^n \sqrt{t_i}A_i\Big|^2\Big\|_{\widehat{Q}}\\
&=&\Big\| \sum_{i=1}^n t_i|A_i|^2\Big\|_{\widehat{Q}}\\
&\leq&\sum_{i=1}^n\Big\|\,t_i\Big|A_i-\sum_{j=1}^n
t_jA_j\Big|^2\Big\|_{\widehat{Q}} + \Big\|\,\Big|
\sum_{j=1}^nt_jA_j\Big|^2\Big\|_{\widehat{Q}}\\
&=& \sum_{i=1}^n t_i\Big\|A_i-\sum_{j=1}^n t_jA_j\Big\|_Q^2 +
\Big\|\sum_{j=1}^nt_jA_j\Big\|_Q^2\,.
\end{eqnarray*}
\end{proof}

%--------------------------------------------------------------------------------------------------%

\begin{remark}
The operators $A_i$ acting on a Hilbert space having the orthogonal
property  $A_i^*A_j=0$ for $1 \leq i\neq j\leq n$ are not uncommon.
For instance, let $(e_i)$ be an orthogonal family (not containing
zero) in ${\mathcal H}$ and define the operators $A_i: {\mathcal H}
\to {\mathcal H}$ by $A_i=\frac{e_i \otimes e_i}{\|e_i\|}$, $1 \leq
i \leq n$. Then $A_i$'s are positive operators in ${\mathbb
B}({\mathcal H})$ with $\|A_i\||_\infty=\|e_i\|$ for all $i$ and
$\|A_iA_j\|_\infty=|\langle e_i, e_j\rangle|$ for all $1 \leq i,j
\leq n$ (for details see \cite{DRA}).
\end{remark}

%--------------------------------------------------------------------------------------------------%

In the setting of Hilbert spaces, the known equality \eqref{0} can
be proved directly. In what follows we show that an extension of it
can also be obtained from equality \eqref{oit}.

\begin{corollary}
Let $x_1, \cdots, x_n\in {\mathcal H}$. Then
\begin{eqnarray*}
\sum_{i=1}^nt_i\Big\|x_i-\sum_{j=1}^n t_jx_j\Big\|^2=
\sum_{i=1}^nt_i\|x_i\|^2 - \Big\|\sum_{j=1}^nt_jx_j\Big\|^2\,.
\end{eqnarray*}
\end{corollary}
\begin{proof}
Let $e$ be a non-zero vector of ${\mathcal H}$ and set $A_i=x_i
\otimes e$. It follows from the elementary properties of rank one
operators and equality \eqref{oit} that
\begin{eqnarray*}
\sum_{i=1}^nt_i\Big\|x_i-\sum_{j=1}^n t_jx_j\Big\|^2 e\otimes e&=&
\sum_{i=1}^nt_i\Big|\Big (x_i-\sum_{j=1}^n
t_jx_j\Big) \otimes e\Big|^2\\
&=&\sum_{i=1}^nt_i\Big|A_i-\sum_{j=1}^n t_jA_j\Big|^2\\&=&
\sum_{i=1}^nt_i|A_i|^2 - \Big|\sum_{j=1}^nt_jA_j\Big|^2\\&=&
\sum_{i=1}^nt_i|x_i\otimes e|^2 - \Big|\Big (\sum_{j=1}^n
t_jx_j\Big )\otimes e\Big|^2\\
&=& \Big (\sum_{i=1}^nt_i\|x_i\|^2 - \Big\|\sum_{j=1}^nt_jx_j\Big
\|^2\ \Big ) e\otimes e\,,
\end{eqnarray*}
from which we conclude the result.
\end{proof}

%--------------------------------------------------------------------------------------------------%

\noindent From now on we restrict ourselves to \eqref{oin}.

\noindent A special case (for $X={\mathbb C}$) of equality \eqref{0}
is
\begin{eqnarray*}
\sum_{i=1}^n\Big|z_i-\frac{1}{n}\sum_{j=1}^n z_j\Big|^2=
\sum_{i=1}^n|z_i|^2 - n\Big|\frac{1}{n}\sum_{j=1}^nz_j\Big|^2\qquad
(z_1, \cdots, z_n\in {\mathbb C})\,.
\end{eqnarray*}
This equality is in turn a special case (when $A=\textrm{diag}(z_1,
\cdots, z_n)$ and $\textrm{tr}([a_{ij}]):=\sum_{i=1}^na_{ii}$\,) of
the next equality concerning the usual normalized trace functional.
\begin{theorem}
Let $\widehat{{\textrm tr}}(A)=\textrm{tr}(A)/\textrm{tr}(I)$ be the
normalized trace on $M_n(\mathbb{C})$. Then
\begin{eqnarray*}
\|A-\widehat{{\textrm tr}}(A)\|_2^2=\|A\|^2_2 -
\textrm{tr}(I)|\widehat{{\textrm tr}}(A)|^2\,.
\end{eqnarray*}
\end{theorem}
\begin{proof}
\begin{eqnarray*}
\left\|A-\widehat{{\textrm tr}}(A)\right\|_2^2&=&
\textrm{tr}\left|A-
\frac{\textrm{tr}(A)}{\textrm{tr}(I)}I\right|^2\\
&=& \textrm{tr}\left(\left(A^*-
\frac{\overline{\textrm{tr}(A)}}{\textrm{tr}(I)}I\right)\left(A-\frac{\textrm{tr}(A)}{\textrm{tr}(I)}I\right)\right)\\
&=& \textrm{tr}(A^*A) -
\frac{\textrm{tr}(A^*)\textrm{tr}(A)}{\textrm{tr}(I)} -
\frac{\overline{\textrm{tr}(A)}\textrm{tr}(A)}{\textrm{tr}(I)}
+ \frac{|\textrm{tr}(A)|^2}{\textrm{tr}(I)^2}\textrm{tr}(I)\\
&=& \textrm{tr}|A|^2 -
\textrm{tr}(I)\left|\frac{\textrm{tr}(A)}{\textrm{tr}(I)}\right|^2
\qquad\qquad(\textrm{by~} \textrm{tr}(A^*)=\overline{\textrm{tr}(A)})\\
&=&\|A\|^2_2 - \textrm{tr}(I)|\widehat{{\textrm tr}}(A)|^2\,.
\end{eqnarray*}
\end{proof}

%--------------------------------------------------------------------------------------------------%

We need the following lemma which can be deduced from
\cite[Lemma~4]{B-K2} and \cite[p.~20]{SIM} (see also
\cite[Lemma~2.1]{H-K-M}).

\begin{lemma}\label{lem2}
Let $A_1, \cdots, A_n$ be positive operators in $C_p$ for some
$p>0$. \begin{enumerate} \item[(i)] If  $ 0 < p < 1$, then
\begin{eqnarray*}
n^{p-1}\sum_{i=1}^n \|A_i\|_p^p \leq \Big \|\sum_{i=1}^n A_i
\Big\|_p^p \leq \sum_{i=1}^n \|A_i\|_p^p\, .
\end{eqnarray*}
\item[(ii)] If $1 \leq p < \infty$, then
 \begin{eqnarray*} \sum_{i=1}^n \|A_i\|_p^p \leq
\Big\|\sum_{i=1}^n A_i \Big\|_p^p \leq n^{p-1} \sum_{i=1}^n
\|A_i\|_p^p\,.
\end{eqnarray*}
\end{enumerate}

\end{lemma}

%--------------------------------------------------------------------------------------------------%

Note that the commutative version of Lemma \ref{lem2} for scalars
follows from the well-known H\"{o}lder inequality (see, e.g.,
\cite[p.~88]{BHA}).

The next theorem is our second main result. It can be regarded as a
generalization of \eqref{4} in inequalities.

%--------------------------------------------------------------------------------------------------%

\begin{theorem}\label{th2} Let $A_1, \cdots, A_n \in C_p$ for some
$p>0$.
\begin{itemize}
\item[(i)] If $0<p<2$, then
$$
\sum_{i=1}^n\Big \|A_i-\frac{1}{n}\sum_{j=1}^n A_j\Big\|_p^p\geq
n^{\frac{p}{2}-1}\sum_{i=1}^n\|A_i\|_p^p - n\Big
\|\frac{1}{n}\sum_{j=1}^nA_j\Big\|_p^p\, .$$
\item[(ii)] If $ 2< p< \infty$, then
$$n^{\frac{p}{2}-1}\sum_{i=1}^n\|A_i\|_p^p -
n\Big\|\frac{1}{n}\sum_{j=1}^nA_j\Big\|_p^p \geq
\sum_{i=1}^n\Big\|A_i-\frac{1}{n}\sum_{j=1}^n A_j\Big\|_p^p\, .$$
\end{itemize}
\end{theorem}

\begin{proof}
Let $0 < p < 2$. Then
\begin{eqnarray*}
\lefteqn{\sum_{i=1}^n\Big\|A_i-\frac{1}{n}\sum_{j=1}^n A_j\Big\|_p^p
+
n\Big\|\frac{1}{n}\sum_{j=1}^nA_j\Big\|_p^p}\\
&=& \sum_{i=1}^n\Big\|\,\Big|A_i-\frac{1}{n}\sum_{j=1}^n
A_j\Big|^2\Big\|_{p/2}^{p/2} +
n\Big\|\,\Big|\frac{1}{n}\sum_{j=1}^nA_j\Big|^2\Big\|_{p/2}^{p/2}\\
&\geq& \Big\|\sum_{i=1}^n\Big|A_i-\frac{1}{n}\sum_{j=1}^n
A_j\Big|^2+ n\Big|\frac{1}{n}\sum_{j=1}^nA_j\Big|^2\Big\|_{p/2}^{p/2}\\
&&\,\,\,\,\,\,\,\,\,\,\, \quad
(\textrm{by the second inequality of Lemma\, \ref{lem2}(i)})\\
&=& \Big\|\sum_{i=1}^n|A_i|^2\Big\|_{p/2}^{p/2}
\qquad (\textrm{by equality \, \eqref{oin}}) \\
&\geq& n^{\frac{p}{2}-1}\sum_{i=1}^n\Big\|\,|A_i|^2\Big\|_{p/2}^{p/2}\\
&&\,\,\,\,\,\,\,\,\,\,\, \quad (\textrm{by the first inequality of
Lemma\, \ref{lem2}(i)})\\
&=& n^{\frac{p}{2}-1}\sum_{i=1}^n \|A_i\|_p^p\,.
\end{eqnarray*}

This proves the first part of the theorem. By a similar argument,
one can prove the second part.
\end{proof}

%--------------------------------------------------------------------------------------------------%

The next theorem may be compared to Theorem \ref{th2}. It can also
be viewed as a variance of \eqref{4}.

\begin{theorem} Let $A_1, \cdots, A_n \in C_p$ for some
$p>0$. Then
$$
\sum_{i=1}^n\Big \|A_i-\frac{1}{n}\sum_{j=1}^n A_j\Big\|_p^2\geq
n^{\frac{2}{p}-1}\sum_{i=1}^n\|A_i\|_p^2 -
n\Big\|\frac{1}{n}\sum_{j=1}^nA_j\Big\|_p^2$$

\noindent for $0<p<2$; and
$$n^{\frac{2}{p}-1}\sum_{i=1}^n\|A_i\|_p^2 -
n\Big\|\frac{1}{n}\sum_{j=1}^nA_j\Big\|_p^2 \geq
\sum_{i=1}^n\Big\|A_i-\frac{1}{n}\sum_{j=1}^n A_j\Big\|_p^2$$
\noindent for $ 2\leq p< \infty$.
\end{theorem}
\begin{proof}
Let $0 < p < 2$. Then
\begin{eqnarray*}
\lefteqn{\sum_{i=1}^n\Big\|A_i-\frac{1}{n}\sum_{j=1}^n A_j
\Big\|_p^2 + n\Big\|\frac{1}{n}\sum_{j=1}^nA_j\Big\|_p^2}\\
&=& \sum_{i=1}^n\Big\|\,\Big|A_i-\frac{1}{n}\sum_{j=1}^n
A_j\Big|^2\Big\|_{p/2} +
n\Big\|\,\Big|\frac{1}{n}\sum_{j=1}^nA_j\Big|^2\Big\|_{p/2}\\
&\leq& \Big\|\sum_{i=1}^n\Big|A_i-\frac{1}{n}\sum_{j=1}^n A_j\Big|^2
+
n\,\Big|\frac{1}{n}\sum_{j=1}^nA_j\Big|^2\Big\|_{p/2}\\
&=& \Big\|\sum_{i=1}^n |A_i|^2\Big\|_{p/2}\; (\textrm {by equality \eqref{oin}})\\
&\leq& \Big(\sum_{i=1}^n \|\,|A_i|^2\,\|_{p/2}^{p/2}\Big)^{2/p} \;
(\textrm{by the second inequality of
Lemma\, \ref{lem2}(i)})\\
&=& \Big(\sum_{i=1}^n \|A_i\|_p^p\Big)^{2/p}\\
&\leq & n^{\frac{2}{p}-1}\sum_{i=1}^n\Big( \|A_i\|_p^p\Big)^{2/p}
\quad
(\textrm{by a scalar version of Lemma \ref{lem2} (ii)})\\
&=& n^{\frac{2}{p}-1}\sum_{i=1}^n \|A_i\|_p^2\,.
\end{eqnarray*}
This proves the first part of the theorem. The second part of the
theorem follows from a similar argument. The second author thanks
the NSU Farquhar College of Arts and Sciences for a Mini-grant.
\end{proof}

%--------------------------------------------------------------------------------------------------%
\textbf{Acknowledgement.} The authors would like to express their
gratitude to the referee for his/her very useful suggestions.

\end{document}